\documentclass[12pt, reqno]{amsart}
\usepackage{amssymb}
\usepackage{eucal}
\usepackage{amsmath}
\usepackage{amscd}
\usepackage[dvips]{color}
\usepackage{multicol}
\usepackage[all]{xy}           
\usepackage{graphicx}
\usepackage{color}
\usepackage{colordvi}
\usepackage{xspace}
\usepackage{tikz}
\usepackage{cite}
\usepackage{ifpdf}
\ifpdf
 \usepackage[colorlinks,final,hyperindex]{hyperref}
\else
 \usepackage[colorlinks,final,hyperindex,hypertex]{hyperref}
\fi

\usepackage[active]{srcltx} 

\begin{document}
\newcommand {\emptycomment}[1]{} 

\baselineskip=14pt
\newcommand{\nc}{\newcommand}
\newcommand{\delete}[1]{}
\nc{\mfootnote}[1]{\footnote{#1}} 
\nc{\todo}[1]{\tred{To do:} #1}

\nc{\mlabel}[1]{\label{#1}}  
\nc{\mcite}[1]{\cite{#1}}  
\nc{\mref}[1]{\ref{#1}}  
\nc{\mbibitem}[1]{\bibitem{#1}} 

\newtheorem{thm}{Theorem}[section]
\newtheorem{lem}[thm]{Lemma}
\newtheorem{cor}[thm]{Corollary}
\newtheorem{pro}[thm]{Proposition}
\newtheorem{ex}[thm]{Example}
\newtheorem{rmk}[thm]{Remark}
\newtheorem{defi}[thm]{Definition}
\newtheorem{pdef}[thm]{Proposition-Definition}
\newtheorem{condition}[thm]{Condition}
\newtheorem{ques}[thm]{Question}

\renewcommand{\labelenumi}{{\rm(\alph{enumi})}}
\renewcommand{\theenumi}{\alph{enumi}}

\nc{\tred}[1]{\textcolor{red}{#1}}
\nc{\tblue}[1]{\textcolor{blue}{#1}}
\nc{\tgreen}[1]{\textcolor{green}{#1}}
\nc{\tpurple}[1]{\textcolor{purple}{#1}}
\nc{\btred}[1]{\textcolor{red}{\bf #1}}
\nc{\btblue}[1]{\textcolor{blue}{\bf #1}}
\nc{\btgreen}[1]{\textcolor{green}{\bf #1}}
\nc{\btpurple}[1]{\textcolor{purple}{\bf #1}}


\nc{\twovec}[2]{\left(\begin{array}{c} #1 \\ #2\end{array} \right )}
\nc{\threevec}[3]{\left(\begin{array}{c} #1 \\ #2 \\ #3 \end{array}\right )}
\nc{\twomatrix}[4]{\left(\begin{array}{cc} #1 & #2\\ #3 & #4 \end{array} \right)}
\nc{\threematrix}[9]{{\left(\begin{matrix} #1 & #2 & #3\\ #4 & #5 & #6 \\ #7 & #8 & #9 \end{matrix} \right)}}
\nc{\twodet}[4]{\left|\begin{array}{cc} #1 & #2\\ #3 & #4 \end{array} \right|}

\nc{\rk}{\mathrm{r}}
\newcommand{\g}{\mathfrak g}
\newcommand{\h}{\mathfrak h}
\newcommand{\pf}{\noindent{$Proof$.}\ }
\newcommand{\frkg}{\mathfrak g}
\newcommand{\frkh}{\mathfrak h}
\newcommand{\Id}{\rm{Id}}
\newcommand{\gl}{\mathfrak {gl}}
\newcommand{\ad}{\mathrm{ad}}
\newcommand{\add}{\frka\frkd}
\newcommand{\frka}{\mathfrak a}
\newcommand{\frkb}{\mathfrak b}
\newcommand{\frkc}{\mathfrak c}
\newcommand{\frkd}{\mathfrak d}
\newcommand {\comment}[1]{{\marginpar{*}\scriptsize\textbf{Comments:} #1}}

\nc{\gensp}{V} 
\nc{\relsp}{\Lambda} 
\nc{\leafsp}{X}    
\nc{\treesp}{\overline{\calt}} 

\nc{\vin}{{\mathrm Vin}}    
\nc{\lin}{{\mathrm Lin}}    

\nc{\gop}{{\,\omega\,}}     
\nc{\gopb}{{\,\nu\,}}
\nc{\svec}[2]{{\tiny\left(\begin{matrix}#1\\
#2\end{matrix}\right)\,}}  
\nc{\ssvec}[2]{{\tiny\left(\begin{matrix}#1\\
#2\end{matrix}\right)\,}} 

\nc{\pll}{\beta}
\nc{\plc}{\epsilon}

\nc{\ass}{{\mathit{Ass}}}
\nc{\lie}{{\mathit{Lie}}}
\nc{\comm}{{\mathit{Comm}}}
\nc{\dend}{{\mathit{Dend}}}
\nc{\zinb}{{\mathit{Zinb}}}
\nc{\tdend}{{\mathit{TDend}}}
\nc{\prelie}{{\mathit{preLie}}}
\nc{\postlie}{{\mathit{PostLie}}}
\nc{\quado}{{\mathit{Quad}}}
\nc{\octo}{{\mathit{Octo}}}
\nc{\ldend}{{\mathit{ldend}}}
\nc{\lquad}{{\mathit{LQuad}}}

 \nc{\adec}{\check{;}} \nc{\aop}{\alpha}
\nc{\dftimes}{\widetilde{\otimes}} \nc{\dfl}{\succ} \nc{\dfr}{\prec}
\nc{\dfc}{\circ} \nc{\dfb}{\bullet} \nc{\dft}{\star}
\nc{\dfcf}{{\mathbf k}} \nc{\apr}{\ast} \nc{\spr}{\cdot}
\nc{\twopr}{\circ} \nc{\tspr}{\star} \nc{\sempr}{\ast}
\nc{\disp}[1]{\displaystyle{#1}}
\nc{\bin}[2]{ (_{\stackrel{\scs{#1}}{\scs{#2}}})}  
\nc{\binc}[2]{ \left (\!\! \begin{array}{c} \scs{#1}\\
    \scs{#2} \end{array}\!\! \right )}  
\nc{\bincc}[2]{  \left ( {\scs{#1} \atop
    \vspace{-.5cm}\scs{#2}} \right )}  
\nc{\sarray}[2]{\begin{array}{c}#1 \vspace{.1cm}\\ \hline
    \vspace{-.35cm} \\ #2 \end{array}}
\nc{\bs}{\bar{S}} \nc{\dcup}{\stackrel{\bullet}{\cup}}
\nc{\dbigcup}{\stackrel{\bullet}{\bigcup}} \nc{\etree}{\big |}
\nc{\la}{\longrightarrow} \nc{\fe}{\'{e}} \nc{\rar}{\rightarrow}
\nc{\dar}{\downarrow} \nc{\dap}[1]{\downarrow
\rlap{$\scriptstyle{#1}$}} \nc{\uap}[1]{\uparrow
\rlap{$\scriptstyle{#1}$}} \nc{\defeq}{\stackrel{\rm def}{=}}
\nc{\dis}[1]{\displaystyle{#1}} \nc{\dotcup}{\,
\displaystyle{\bigcup^\bullet}\ } \nc{\sdotcup}{\tiny{
\displaystyle{\bigcup^\bullet}\ }} \nc{\hcm}{\ \hat{,}\ }
\nc{\hcirc}{\hat{\circ}} \nc{\hts}{\hat{\shpr}}
\nc{\lts}{\stackrel{\leftarrow}{\shpr}}
\nc{\rts}{\stackrel{\rightarrow}{\shpr}} \nc{\lleft}{[}
\nc{\lright}{]} \nc{\uni}[1]{\tilde{#1}} \nc{\wor}[1]{\check{#1}}
\nc{\free}[1]{\bar{#1}} \nc{\den}[1]{\check{#1}} \nc{\lrpa}{\wr}
\nc{\curlyl}{\left \{ \begin{array}{c} {} \\ {} \end{array}
    \right .  \!\!\!\!\!\!\!}
\nc{\curlyr}{ \!\!\!\!\!\!\!
    \left . \begin{array}{c} {} \\ {} \end{array}
    \right \} }
\nc{\leaf}{\ell}       
\nc{\longmid}{\left | \begin{array}{c} {} \\ {} \end{array}
    \right . \!\!\!\!\!\!\!}
\nc{\ot}{\otimes} \nc{\sot}{{\scriptstyle{\ot}}}
\nc{\otm}{\overline{\ot}}
\nc{\ora}[1]{\stackrel{#1}{\rar}}
\nc{\ola}[1]{\stackrel{#1}{\la}}
\nc{\pltree}{\calt^\pl}
\nc{\epltree}{\calt^{\pl,\NC}}
\nc{\rbpltree}{\calt^r}
\nc{\scs}[1]{\scriptstyle{#1}} \nc{\mrm}[1]{{\rm #1}}
\nc{\dirlim}{\displaystyle{\lim_{\longrightarrow}}\,}
\nc{\invlim}{\displaystyle{\lim_{\longleftarrow}}\,}
\nc{\mvp}{\vspace{0.5cm}} \nc{\svp}{\vspace{2cm}}
\nc{\vp}{\vspace{8cm}} \nc{\proofbegin}{\noindent{\bf Proof: }}
\nc{\proofend}{$\blacksquare$ \vspace{0.5cm}}
\nc{\freerbpl}{{F^{\mathrm RBPL}}}
\nc{\sha}{{\mbox{\cyr X}}}  
\nc{\ncsha}{{\mbox{\cyr X}^{\mathrm NC}}} \nc{\ncshao}{{\mbox{\cyr
X}^{\mathrm NC,\,0}}}
\nc{\shpr}{\diamond}    
\nc{\shprm}{\overline{\diamond}}    
\nc{\shpro}{\diamond^0}    
\nc{\shprr}{\diamond^r}     
\nc{\shpra}{\overline{\diamond}^r}
\nc{\shpru}{\check{\diamond}} \nc{\catpr}{\diamond_l}
\nc{\rcatpr}{\diamond_r} \nc{\lapr}{\diamond_a}
\nc{\sqcupm}{\ot}
\nc{\lepr}{\diamond_e} \nc{\vep}{\varepsilon} \nc{\labs}{\mid\!}
\nc{\rabs}{\!\mid} \nc{\hsha}{\widehat{\sha}}
\nc{\lsha}{\stackrel{\leftarrow}{\sha}}
\nc{\rsha}{\stackrel{\rightarrow}{\sha}} \nc{\lc}{\lfloor}
\nc{\rc}{\rfloor}
\nc{\tpr}{\sqcup}
\nc{\nctpr}{\vee}
\nc{\plpr}{\star}
\nc{\rbplpr}{\bar{\plpr}}
\nc{\sqmon}[1]{\langle #1\rangle}
\nc{\forest}{\calf}
\nc{\altx}{\Lambda_X} \nc{\vecT}{\vec{T}} \nc{\onetree}{\bullet}
\nc{\Ao}{\check{A}}
\nc{\seta}{\underline{\Ao}}
\nc{\deltaa}{\overline{\delta}}
\nc{\trho}{\tilde{\rho}}

\nc{\rpr}{\circ}
\nc{\dpr}{{\tiny\diamond}}
\nc{\rprpm}{{\rpr}}

\nc{\mmbox}[1]{\mbox{\ #1\ }} \nc{\ann}{\mrm{ann}}
\nc{\Aut}{\mrm{Aut}} \nc{\can}{\mrm{can}}
\nc{\twoalg}{{two-sided algebra}\xspace}
\nc{\colim}{\mrm{colim}}
\nc{\Cont}{\mrm{Cont}} \nc{\rchar}{\mrm{char}}
\nc{\cok}{\mrm{coker}} \nc{\dtf}{{R-{\rm tf}}} \nc{\dtor}{{R-{\rm
tor}}}
\renewcommand{\det}{\mrm{det}}
\nc{\depth}{{\mrm d}}
\nc{\Div}{{\mrm Div}} \nc{\End}{\mrm{End}} \nc{\Ext}{\mrm{Ext}}
\nc{\Fil}{\mrm{Fil}} \nc{\Frob}{\mrm{Frob}} \nc{\Gal}{\mrm{Gal}}
\nc{\GL}{\mrm{GL}} \nc{\Hom}{\mrm{Hom}} \nc{\hsr}{\mrm{H}}
\nc{\hpol}{\mrm{HP}} \nc{\id}{\mrm{id}} \nc{\im}{\mrm{im}}
\nc{\incl}{\mrm{incl}} \nc{\length}{\mrm{length}}
\nc{\LR}{\mrm{LR}} \nc{\mchar}{\rm char} \nc{\NC}{\mrm{NC}}
\nc{\mpart}{\mrm{part}} \nc{\pl}{\mrm{PL}}
\nc{\ql}{{\QQ_\ell}} \nc{\qp}{{\QQ_p}}
\nc{\rank}{\mrm{rank}} \nc{\rba}{\rm{RBA }} \nc{\rbas}{\rm{RBAs }}
\nc{\rbpl}{\mrm{RBPL}}
\nc{\rbw}{\rm{RBW }} \nc{\rbws}{\rm{RBWs }} \nc{\rcot}{\mrm{cot}}
\nc{\rest}{\rm{controlled}\xspace}
\nc{\rdef}{\mrm{def}} \nc{\rdiv}{{\rm div}} \nc{\rtf}{{\rm tf}}
\nc{\rtor}{{\rm tor}} \nc{\res}{\mrm{res}} \nc{\SL}{\mrm{SL}}
\nc{\Spec}{\mrm{Spec}} \nc{\tor}{\mrm{tor}} \nc{\Tr}{\mrm{Tr}}
\nc{\mtr}{\mrm{sk}}

\nc{\ab}{\mathbf{Ab}} \nc{\Alg}{\mathbf{Alg}}
\nc{\Algo}{\mathbf{Alg}^0} \nc{\Bax}{\mathbf{Bax}}
\nc{\Baxo}{\mathbf{Bax}^0} \nc{\RB}{\mathbf{RB}}
\nc{\RBo}{\mathbf{RB}^0} \nc{\BRB}{\mathbf{RB}}
\nc{\Dend}{\mathbf{DD}} \nc{\bfk}{{\bf k}} \nc{\bfone}{{\bf 1}}
\nc{\base}[1]{{a_{#1}}} \nc{\detail}{\marginpar{\bf More detail}
    \noindent{\bf Need more detail!}
    \svp}
\nc{\Diff}{\mathbf{Diff}} \nc{\gap}{\marginpar{\bf
Incomplete}\noindent{\bf Incomplete!!}
    \svp}
\nc{\FMod}{\mathbf{FMod}} \nc{\mset}{\mathbf{MSet}}
\nc{\rb}{\mathrm{RB}} \nc{\Int}{\mathbf{Int}}
\nc{\Mon}{\mathbf{Mon}}
\nc{\remarks}{\noindent{\bf Remarks: }}
\nc{\OS}{\mathbf{OS}} 
\nc{\Rep}{\mathbf{Rep}}
\nc{\Rings}{\mathbf{Rings}} \nc{\Sets}{\mathbf{Sets}}
\nc{\DT}{\mathbf{DT}}

\nc{\BA}{{\mathbb A}} \nc{\CC}{{\mathbb C}} \nc{\DD}{{\mathbb D}}
\nc{\EE}{{\mathbb E}} \nc{\FF}{{\mathbb F}} \nc{\GG}{{\mathbb G}}
\nc{\HH}{{\mathbb H}} \nc{\LL}{{\mathbb L}} \nc{\NN}{{\mathbb N}}
\nc{\QQ}{{\mathbb Q}} \nc{\RR}{{\mathbb R}} \nc{\BS}{{\mathbb{S}}} \nc{\TT}{{\mathbb T}}
\nc{\VV}{{\mathbb V}} \nc{\ZZ}{{\mathbb Z}}


\nc{\calao}{{\mathcal A}} \nc{\cala}{{\mathcal A}}
\nc{\calc}{{\mathcal C}} \nc{\cald}{{\mathcal D}}
\nc{\cale}{{\mathcal E}} \nc{\calf}{{\mathcal F}}
\nc{\calfr}{{{\mathcal F}^{\,r}}} \nc{\calfo}{{\mathcal F}^0}
\nc{\calfro}{{\mathcal F}^{\,r,0}} \nc{\oF}{\overline{F}}
\nc{\calg}{{\mathcal G}} \nc{\calh}{{\mathcal H}}
\nc{\cali}{{\mathcal I}} \nc{\calj}{{\mathcal J}}
\nc{\call}{{\mathcal L}} \nc{\calm}{{\mathcal M}}
\nc{\caln}{{\mathcal N}} \nc{\calo}{{\mathcal O}}
\nc{\calp}{{\mathcal P}} \nc{\calq}{{\mathcal Q}} \nc{\calr}{{\mathcal R}}
\nc{\calt}{{\mathcal T}} \nc{\caltr}{{\mathcal T}^{\,r}}
\nc{\calu}{{\mathcal U}} \nc{\calv}{{\mathcal V}}
\nc{\calw}{{\mathcal W}} \nc{\calx}{{\mathcal X}}
\nc{\CA}{\mathcal{A}}

\nc{\fraka}{{\mathfrak a}} \nc{\frakB}{{\mathfrak B}}
\nc{\frakb}{{\mathfrak b}} \nc{\frakd}{{\mathfrak d}}
\nc{\oD}{\overline{D}}
\nc{\frakF}{{\mathfrak F}} \nc{\frakg}{{\mathfrak g}}
\nc{\frakm}{{\mathfrak m}} \nc{\frakM}{{\mathfrak M}}
\nc{\frakMo}{{\mathfrak M}^0} \nc{\frakp}{{\mathfrak p}}
\nc{\frakS}{{\mathfrak S}} \nc{\frakSo}{{\mathfrak S}^0}
\nc{\fraks}{{\mathfrak s}} \nc{\os}{\overline{\fraks}}
\nc{\frakT}{{\mathfrak T}}
\nc{\oT}{\overline{T}}
\nc{\frakX}{{\mathfrak X}} \nc{\frakXo}{{\mathfrak X}^0}
\nc{\frakx}{{\mathbf x}}
\nc{\frakTx}{\frakT}      
\nc{\frakTa}{\frakT^a}        
\nc{\frakTxo}{\frakTx^0}   
\nc{\caltao}{\calt^{a,0}}   
\nc{\ox}{\overline{\frakx}} \nc{\fraky}{{\mathfrak y}}
\nc{\frakz}{{\mathfrak z}} \nc{\oX}{\overline{X}}

\font\cyr=wncyr10

\nc{\redtext}[1]{\textcolor{red}{#1}}

\title{\bf On the tensor product of two oriented quantum algebras}
\author{Tianshui Ma\textsuperscript{*}}
\address{School of Mathematics and Information Science, Henan Normal University, Xinxiang 453007, China}
\email{matianshui@yahoo.com}

\author{Haiyan Yang}
\address{School of Mathematics and Information Science, Henan Normal University, Xinxiang 453007, China}
\email{yhy3023551288@163.com}

\author{Tao Yang}
\address{Department of Mathematics, Nanjing Agricultural University, Nanjing 210095, China}
\email{yangtao19841113@gmail.com}

\thanks{\textsuperscript{*}Corresponding author}

\begin{abstract}                                                                        
 In this paper, we give the oriented quantum algebra (abbr. OQA) structures on the tensor product of two different OQAs by using Chen's weak $\mathfrak{R}$-matrix in [J. Algebra 204(1998):504-531]. As a special case, the OQA structures on the tensor product of an OQA with itself are provided, which are different from Radford's results in [J. Knot Theory Ramifications 16(2007):929-957].
\end{abstract}

\subjclass[2010]{16T05}

\keywords{Quantum Yang-Baxter equation; Quasitriangular Hopf algebra; Oriented quantum algebra.}

 \maketitle

\tableofcontents

\numberwithin{equation}{section}

\tableofcontents
\numberwithin{equation}{section}
\allowdisplaybreaks

\section{Introduction}
 It is well known that the quantum Yang-Baxter equation has wide applications in mathematics and physics and an $\mathfrak{R}$-matrix gives rise to a solution for the quantum Yang-Baxter equation. Algebraic structures related to the quantum Yang-Baxter equation have been extensively studied in the literature recently, see Refs\cite{Ch,MLW,MLL1,MZ,MW,Ra1}. The notion of oriented quantum algebra (abbr. OQA) (this structure can provide the solution for the quantum Yang-Baxter equation) was introduced and motivated from a topological point of view. OQAs account for most known regular isotopy invariants of oriented links, thus OQAs are important for that reason. Quasitriangular Hopf algebras are the important examples of OQAs.

 So it is important to construct OQAs for obtaining new invariants. This is the main motivation to write this paper.

 In \cite{Ra2}, Radford gave the OQA structure on the tensor product of an OQA with itself, via an algebra isomorphism of the quantum double $D(A)$ with $A\otimes A$ under certain condition. In \cite{MW}, Ma and Wang got a generalization of Radford's results just by observing and analyzing his construction. Let $H$ and $H'$ be two quasitriangular Hopf algebras, Chen got the quasitriangular structures of bicrossed coproducts $H\bowtie^ r H'$ by using a so-called weak $\mathfrak{R}$-matrix of $(H, H')$( \cite{Ch}). In this paper, we follow Chen's idea about weak $\mathfrak{R}$-matrix and investigate a new approach to obtain some OQA structures.

 In this paper, we present a class of OQA structures on the tensor product of two OQAs, which is a generalization of the Chen's Theorem (see \cite[Theorem 2.2]{Ch}) and is different from the Radford's Theorem (see \cite[Theorem 4.1]{Ra2}). In Section 3, the OQA structure on the tensor product of two OQAs is derived under a suitable condition (see Theorem \ref{thm:3.6}). Then we can obtain the OQA structure on the tensor product of an OQA with itself (see Theorem \ref{thm:3.7}), which is different from \cite[Theorem 4.1]{Ra2} (see Remarks \ref{rmk:3.8} and \ref{rmk:4.4}). And also, our result generalizes the OQA construction implicit in Chen's theorem \cite[Theorem 2.2]{Ch} (see Theorem \ref{thm:3.10}).  Finally, in Section 4, several nontrivial concrete examples are given.

\section{Preliminaries}

 Throughout the paper, we follow the definitions and terminologies in \cite{Ch,Ra1,Ra2} and all algebraic systems are over a field $K$. Let $C$ be a coalgebra. Then we use the Heyneman-Sweedler's notation for the comultiplication $\Delta(c)=c_{1}\otimes c_{2}$ for any $c\in C$. Given a $K$-space $M$, we write $id_M$ for the identity map on $M$.  Let $(H, \mu, \eta)$ be an algebra. The opposite algebra
 $(H, \mu^{op}, \eta)$, denoted by $H^{op}$, has multiplication described in terms of that for $H$ by $\mu^{op}(a\otimes b)=ba$ for all $a, b\in H$.

 Let $r=\sum r_{i}\otimes r^{i}\in A\otimes A$. We represent sums of tensors formally as a single tensor. For example we write $r=r_{i}\otimes r^{i}=r_{j}\otimes r^{j}=r_{k}\otimes r^{k}=\cdots$. When $r$ is invertible, we will frequently write $r^{-1}=R=R_{l}\otimes R^{l}=R_{m}\otimes R^{m}=R_{n}\otimes R^{n}=\cdots$.  We define $r_{12}=r_{i}\otimes r^{i}\otimes 1, r_{13}=r_{i}\otimes 1\otimes r^{i}$ and $r_{23}=1\otimes r_{i}\otimes r^{i}$.

\smallskip

 Next we recall \cite{Ch,Ra2,MLW} some basic definitions and results which will be used later.

\begin{defi}\label{de:1.1} {\rm Suppose that $r=r_{i}\otimes r^{i}\in H\otimes H$ is invertible and
 $D,U$ are commuting algebra automorphisms of $H$. Then $(H, r, D, U)$ is an {\bf oriented quantum algebra} (abbr. OQA) if the following conditions hold:
 \begin{eqnarray}
 &&(D\otimes I_H)(r^{-1}) \hbox{~and~} (I_H\otimes U)(r) \hbox{~are inverses in~} H\otimes H^{op}; \label{eq:2.1}\\
 &&(D\otimes D)(r)=r=(U\otimes U)(r); \label{eq:2.2}\\
 &&r_{12}r_{13}r_{23}=r_{23}r_{13}r_{12}. \label{eq:2.3}
 \end{eqnarray}

 Formally writing $r=r_{i}\otimes r^{i}=r_{j}\otimes r^{j}=r_{k}\otimes r^{k}$ and $r^{-1}=R=R_{l}\otimes R^{l}$, we can
 reformulate the preceding axioms:
 \begin{eqnarray}
 &&D(R_{l})r_{i}\otimes U(r^{i})R^{l}=1\otimes 1=r_{i}D(R_{l})\otimes R^{l}U(r^{i});\label{eq:2.4}\\
 &&D(r_{i})\otimes D(r^{i})=r=U(r_{i})\otimes U(r^{i});\label{eq:2.5}\\
 &&r_{i}r_{j}\otimes r^{i}r_{k}\otimes r^{j}r^{k}=r_{j}r_{i}\otimes r_{k}r^{i}\otimes r^{k}r^{j}.\label{eq:2.6}
 \end{eqnarray}

 The condition Eq.(\ref{eq:2.3}) has the following equivalent form:
 \begin{equation}\label{eq:2.7}
 r_{j}\otimes r_{i}\otimes r^{j} r^{i}=R_{l}r_{j} r_{i}\otimes R^{l}r_{k}r^{i}\otimes r^{k}r^{j}.
 \end{equation}}
\end{defi}

\begin{rmk}\label{rmk:2.2}
 Suppose that $(H, r, D, U)$ is an OQA. Then $U\otimes D$ is an algebra automorphism of $H\otimes H^{op}$. Since $D(R_{l})\otimes R^{l}$ and $r_{i}\otimes U(r^{i})$ are inverses in $H\otimes H^{op}$, it follows that $(H,r,U,D)$ is an OQA. Suppose that $(H', r', D', U')$ is also an OQA and let $\tau_{H, H'}: H\rightarrow H'$ be defined by $\tau_{H,H'}(h\otimes h')=h'\otimes h$ for all $h\in H, h'\in H'$. Then $(H\otimes H', r'', D\otimes D', U\otimes U')$ is an OQA, called the tensor product of $(H, r, D, U)$ and $(H', r', D', U')$, where write
 $r''=(id_H\otimes \tau_{H,H'}\otimes id_{H'})(r\otimes r')$.
\end{rmk}

\begin{thm}\label{thm:2.3} Suppose that $(H, r, D, U)$ is an OQA and write $r=r_{i}\otimes
 r^{i}=r_{j}\otimes r^{j}$,  $r^{-1}=R=R_{l}\otimes R^{l}=R_{m}\otimes R^{m}$. Then $(H\otimes H,  \tilde{\alpha}, \tilde{D}, \tilde{U})$ is an OQA,
 where
 \begin{equation}\label{eq:2.8}
 \tilde{\alpha}=R^{l}r_{i}\otimes R^m r_{j}\otimes r^i r^j \otimes R_{l}R_{m},
 \end{equation}
 $\tilde{D}=D\otimes D$ and $\tilde{U}=U\otimes U$.
\end{thm}

\begin{defi}\label{de:2.4} {\rm A {\bf quasitriangular Hopf algebra} is a pair $(H,r)$, where $H$ is a Hopf algebra, $r$ is an
 element in $H\otimes H$ satisfying the following conditions:
 \begin{eqnarray}
 &&(\Delta\otimes id)(r)=r_{j}\otimes r_{i}\otimes r^{j} r^{i};\label{eq:2.9}\\
 &&(id\otimes \Delta)(r)=r_{j} r_{i}\otimes r^{i}\otimes r^{j};\label{eq:2.10}\\
 &&\Delta^{cop}(h)r=r\Delta(h);\label{eq:2.11}\\
 &&\varepsilon(r_{i})r^{i}=r_{i}\varepsilon(r^{i})=1,\label{eq:2.12}
 \end{eqnarray}
 where $r=r_{i}\otimes r^{i}=r_{j}\otimes r^{j}$.}
\end{defi}

\begin{rmk}\label{rmk:2.5}
 In this case, $r$ is called a {\bf universal $\mathfrak{R}$-matrix of $H$} and $r$ is invertible. If  $(H, r)$ is a quasitriangular Hopf algebra, then the antipode $S_H$ is bijective and $(H, r, id_H, S_H^{-2})$ is an OQA.
\end{rmk}

\begin{defi}\label{de:2.6} {\rm Let $H, H'$ be bialgebras. An invertible element $r\in H\otimes H'$ is a {\bf weak $\mathfrak{R}$-matrix of $(H, H')$} if the following conditions are satisfied:
 \begin{eqnarray}
 &&(\Delta \otimes id)(r)= r_{i}\otimes r_{j}\otimes r^{i} r^{j};\label{eq:2.13}\\
 &&(id\otimes \Delta)(r)= r_{i} r_{j}\otimes r^{j} \otimes r^{i},\label{eq:2.14}
 \end{eqnarray}
 where $r=r_{i}\otimes r^{i}=r_{j}\otimes r^{j}$.}
\end{defi}

\begin{rmk}\label{rmk:2.7}
 If $H$ and $H'$ are two Hopf algebras with bijective antipodes and  $r=r_{i}\otimes r^{i}\in H\otimes H'$ is a weak $\mathfrak{R}$-matrix of $(H, H')$, then
 \begin{equation}\label{eq:2.15}
 r^{-1}=S_H(r_{i})\otimes r^{i}=r_{i}\otimes S_{H'}^{-1}(r^{i})
 \end{equation}
 and
 \begin{equation}\label{eq:2.16}
 r=(S_H\otimes S_{H'})(r).
 \end{equation}
\end{rmk}

\begin{thm}\label{thm:2.8}
 Let $H$ and $H'$ be two Hopf algebras and  $r=r_{i}\otimes r^{i}\in H\otimes H'$ a weak $\mathfrak{R}$-matrix of $(H, H')$ ( write $r^{-1}=R=R_{l}\otimes R^{l}$), then $H\bowtie^ r H'$ ($=H\otimes H'$ as an algebra) with the following coproduct $\bar{\Delta}$ and antipode $\bar{S}$ is a Hopf algebra, where
 $$
 \bar{\Delta}(h\otimes h')=h_1\otimes r^{i}h'_1 R^{l}\otimes r_{i}h_2 R_{l}\otimes h'_2
 $$
 and
 $$
 \bar{S}(h\otimes h')=r^{-1}(S_H(h)\otimes S_{H'}(h'))r,
 $$
 $\forall h\in H, h'\in H'$. We call this Hopf algebra {\bf bicrossed coproduct Hopf algebra}.
\end{thm}

\begin{thm}\label{thm:2.9}
 Assume that $(H, p)$ and $(H', p')$ are quasitriangular Hopf algebras, where $p=p_{i}\otimes p^{i}\in H\otimes H,~p'=
 p'_{j}\otimes p'^{j}\in H'\otimes H'$.  Let $r=r_{k}\otimes r^{k}\in H\otimes H'$ be a weak $\mathfrak{R}$-matrix of $(H, H')$  ( write $r^{-1}=R=R_{l}\otimes R^{l}$). Set $[p, p']=r_{k}p_{i}\otimes p'_{j}R^{l}\otimes p^{i}R_{l}\otimes r^{k}p'^{j}$, then $(H\bowtie^r H', [p, p'])$ is a quasitriangular Hopf algebra.
\end{thm}

\section{New oriented quantum algebra structures}
 In this section, we will give the OQA structure on the tensor product of two OQAs.

\subsection{OQA nonuple}

\begin{defi}\label{de:3.1} {\rm
 Let $H$ and $H'$ be algebras. Suppose that the elements $p\in H\otimes H,~p'\in H'\otimes H'$ and $r\in H\otimes H'$ are invertible.  Let $D, U: H\rightarrow H$ be commuting algebra automorphisms, $D',U': H'\rightarrow H'$ be commuting algebra automorphisms, and $p, p', r, D, U, D', U'$ satisfy the following conditions:
 \begin{eqnarray}
 &&D(R_{l})r_{k}\otimes U'(r^{k})R^{l}=1_{H}\otimes 1_{H'}=r_{k}D(R_{l})\otimes R^{l}U'(r^{k});\label{eq:3.1}\\
 &&(D\otimes D')(r)=r=(U\otimes U')(r);\label{eq:3.2}\\
 &&p_{i}r_{k}\otimes p^{i}r_{s}\otimes r^{k}r^{s}=r_{k}p_{i}\otimes r_{s}p^{i}\otimes r^{s}r^{k};\label{eq:3.3}\\
 &&r_{k}r_{s}\otimes r^{k}p'_{j}\otimes r^{s}p'^{j}=r_{s}r_{k}\otimes p'_{j}r^{k}\otimes p'^{j}r^{s},\label{eq:3.4}
 \end{eqnarray}
 where $p=p_{i}\otimes p^{i},~p'=p'_{j}\otimes p'^{j}$, $r=r_{k}\otimes r^{k}=r_{s}\otimes r^{s}$ and $r^{-1}=R=R_{l}\otimes R^{l}$. Then we call $(H, H', p, p', r, D, U, D', U')$ an {\bf OQA nonuple}.}
\end{defi}

\begin{rmk}\label{rmk:3.2}
 (1) When $(H,p,D,U)$ is an OQA, $(H, H, p, p, p,$  $D, U, D, U)$ is an OQA nonuple.\newline
 \indent{\phantom{\bf Remarks}} (2) If $(H,H',p,p',r,D,U,D',U')$ is an OQA nonuple, then $(H,H',p,p',r,U,D,U',$ $D')$ is
 also an OQA nonuple because $D,U$ and $D',U'$ are commuting algebra automorphisms, respectively.\newline
 \indent{\phantom{\bf Remarks}} (3) By Eqs.(\ref{eq:3.3}) and (\ref{eq:3.4}), we can get the following identities:
 \begin{eqnarray}
 &&r_{s}\otimes r_{k}\otimes r^{s} r^{k}=P_{l}  r_{s} p_{i}\otimes P^{l} r_{k} p^{i}\otimes r^{k} r^{s};\label{eq:3.5}\\
 &&r_{k}\otimes p'_{j}\otimes r^{k} p'^{j}=R_{l} r_{s} r_{k}\otimes R^{l} p'_{j} r^{k}\otimes p'^{j} r^{s};\label{eq:3.6}\\
 &&R^{l}\otimes p_{i}\otimes R_{l}p^{i}=r^{k} R^{m}R^{l}\otimes r_{k} p_{i} R_{l}\otimes p^{i} R_{m};\label{eq:3.7}\\
 &&R^{m}\otimes R^{l}\otimes R_{m} R_{l}=P'_{n} R^{m} p'_{j}\otimes P'^{n} R^{l} p'^{j}\otimes R_{l} R_{m}.\label{eq:3.8}
 \end{eqnarray}
\end{rmk}

\begin{pro}\label{pro:3.3}
 Assume that $(H, p)$, $(H', p')$ are quasitriangular Hopf algebras, where $p=p_{i}\otimes p^{i}\in H\otimes H,~p'=p'_{j}\otimes p'^{j}\in H'\otimes H'$, and $r=r_{k}\otimes r^{k}=r_{s}\otimes r^{s}\in H\otimes H'$ is a weak $\mathfrak{R}$-matrix of $(H, H')$. Then Eqs.(\ref{eq:3.3}) and (\ref{eq:3.4}) hold.
\end{pro}

\begin{proof} We only check Eq.(\ref{eq:3.4}) as follows:
 \begin{eqnarray*}
 r_{s}r_{k}\otimes p'_{j}r^{k}\otimes p'^{j}r^{s}
 &\stackrel{(\ref{eq:2.14})}{=}&r_{k}\otimes p'_{j}{r^{k}}_1\otimes p'^{j}{r^{k}}_2\\
 &\stackrel{(\ref{eq:2.11})}{=}&r_{k}\otimes {r^{k}}_2p'_{j}\otimes {r^{k}}_1p'^{j}\\
 &\stackrel{(\ref{eq:2.14})}{=}&r_{k}r_{s}\otimes r^{k}p'_{j}\otimes r^{s}p'^{j},
 \end{eqnarray*}
 finishing the proof.     \end{proof}

\begin{ex}\label{ex:3.4}
 Let $A=M_n(K)$ be the algebra of $n\times n$ matrices. For  $1\leq i, j\leq n$, let $E_{ij}\in M_n(K)$ be the $n\times n$ matrix which has a single non-zero entry, the value $1$ located in the $i$th row and $j$th column. Then $\{E_{ij}\}_{1\leq i, j\leq n}$ is the standard basis for $M_n(K)$ and $E_{ij}E_{lm}=\delta_{jl}E_{im}$ for all $1\leq i, j, l, m\leq n$.  Let $n\geq 2$, $a\in K^*$ satisfy $a^2\neq 1$.  Set
\begin{eqnarray*}
 &p_{a, n}=&\sum_{1\leq i< j\leq n}(a-a^{-1})E_{ij}\otimes E_{ji} +\sum_{i=1}^n aE_{ii}\otimes E_{ii}\\
 &&+\sum_{1\leq i< j\leq n} (E_{ii}\otimes E_{jj}+ E_{jj}\otimes E_{ii}).
 \end{eqnarray*}
 Then $(M_n(K), p_{a, n}, f, f)$ is an OQA, where
 $$
 f(E_{ij})=a^{i-j}E_{ij}
 $$
 for all $1\leq i, j\leq n$ (see \cite{KR2}).

 It follows that  $(H=M_2(K), p, t, t)$ and $(H'=M_3(K), p', t', t')$ are OQAs, where $t=t'=f$, and
 \begin{eqnarray*}
  &p=&(a-a^{-1})E_{12}\otimes E_{21} +\sum_{i=1}^2 aE_{ii}\otimes E_{ii}\\
 &&+E_{11}\otimes E_{22}+ E_{22}\otimes E_{11},\\
 &p'=&\sum_{1\leq i< j\leq 3}(a-a^{-1})E'_{ij}\otimes E'_{ji} +\sum_{i=1}^3 aE'_{ii}\otimes E'_{ii}\\
 &&+\sum_{1\leq i< j\leq 3}(E'_{ii}\otimes E'_{jj}+ E'_{jj} \otimes E'_{ii}).
 \end{eqnarray*}

 {\it Case I:} If
  $$
  r=E_{11}\otimes E'_{11}+E_{22}\otimes E'_{22}-E_{22}\otimes E'_{33}-E_{22}\otimes E'_{11}-E_{11}\otimes E'_{22}-E_{11}\otimes E'_{33}
  $$
 and its inverse $R=r$, then it is straight to check that $(H, H', p, p', r, t, t, t', t')$ is  an OQA nonuple.

 {\it Case II:} If
 \begin{eqnarray*}
 &r=&aE_{11}\otimes E'_{11}+aE_{22}\otimes E'_{22}+E_{22}\otimes E'_{33}+E_{22}\otimes E'_{11}+E_{11}\otimes E'_{22}\\
 &&+E_{11}\otimes E'_{33}+(a-a^{-1})E_{12}\otimes E'_{21}
 \end{eqnarray*}
 and its inverse is given by
  \begin{eqnarray*}
 &R=&a^{-1}E_{11}\otimes E'_{11}+a^{-1}E_{22}\otimes E'_{22}+E_{22}\otimes E'_{33}+E_{22}\otimes E'_{11}+E_{11}\otimes E'_{22}\\
 && +E_{11}\otimes E'_{33}+(a^{-1}-a)E_{12}\otimes E'_{21},
 \end{eqnarray*}
 then a straightforward verification shows that $(H, H', p, p', r, t, t, t', t')$ is  an OQA nonuple.
\end{ex}

\begin{thm}\label{thm:3.5}
 Suppose that $(H, p, D, U)$, $(H', p', D', U')$ are OQAs and $(H, H', p, p', r,D, U, D',$ $ U')$ , $(H, H', p, p', q, D, U, D', U')$ are two OQA nonuples, write $p=p_{i}\otimes p^{i}, p'=p'_{j}\otimes p'^{j}, P=p^{-1}=P_{l}\otimes P^{l}, P'=p'^{-1}=P'_{m}\otimes P'^{m}, r=r_{k}\otimes r^{k}, R=r^{-1}=R_{n}\otimes R^{n}, q=q_{s}\otimes q^{s}, Q=q^{-1}=Q_{t}\otimes Q^{t}$, and $p, p', r, q$ satisfy the following conditions:
 \begin{eqnarray}
 && Q^{t} p'_{j}\otimes Q_{t} r_{k}\otimes p'^{j} r^{k}=p'_{j} Q^{t}\otimes r_{k} Q_{t}\otimes r^{k} p'^{j};\label{eq:3.9}\\
 && r_{k} p_{i}\otimes r^{k} Q^{t}\otimes p^{i} Q_{t}=p_{i} r_{k}\otimes Q^{t} r^{k}\otimes Q_{t} p^{i}.\label{eq:3.10}
 \end{eqnarray}
 Then ($H\otimes H'$, $\tilde{\alpha}$, $\tilde{D}$, $\tilde{U}$) is an OQA, where
 $$
 ~~~~~\tilde{\alpha}=r_{k} p_{i}\otimes p'_{j} Q^{t}\otimes p^{i} Q_{t}\otimes r^{k} p'^{j},~\tilde{D}=D\otimes D',~\tilde{U}=U\otimes U';
 $$
 $$
 \hbox{and~~}\tilde{\alpha}^{-1}=P_{l} R_{n}\otimes q^{s} P'_{m}\otimes q_{s} P^{l}\otimes P'^{m} R^{n}.
 $$
\end{thm}

\begin{proof}
 Clearly $\tilde{\alpha}^{-1}=P_{l} R_{n}\otimes q^{s} P'_{m}\otimes q_{s} P^{l}\otimes P'^{m} R^{n}$. Next we will prove that Eqs.(\ref{eq:2.1}) and (\ref{eq:2.3}) hold for ($H\otimes H'$, $\tilde{\alpha}$, $\tilde{D}$, $\tilde{U}$).

 Firstly, Eq.(\ref{eq:2.1}) can be check as follows:
 \begin{eqnarray*}
 \tilde{D}({\tilde{\alpha}^{-1}}_{j}){\tilde{\alpha}}_{i}\otimes \tilde{U}(\tilde{\alpha}^{i})\tilde{\alpha}^{-1j}
 &=&D(P_{l})D(R_{n})r_{k}p_{i}\otimes D'(q^{s})\underline{D'(P'_{m})p'_{j}}Q^{t}\\
 &&\otimes U(p^{i})U(Q_{t})q_{s}P^{l}\otimes U'(r^{k})\underline{U'(p'^{j})P'^{m}}R^{n}\\
 &\stackrel{(\ref{eq:2.1})}{=}&D(P_{l})\underline{D(R_{n})r_{k}}p_{i}\otimes D'(q^{s})Q^{t}\\
 &&\otimes U(p^{i})U(Q_{t})q_{s}P^{l}\otimes \underline{U'(r^{k})R^{n}}\\
 &\stackrel{(\ref{eq:3.1})}{=}&D(P_{l})p_{i}\otimes \underline{D'(q^{s})Q^{t}}\otimes U(p^{i})\underline{U(Q_{t})q_{s}}P^{l}\otimes 1_{H'}\\
 &\stackrel{(\ref{eq:3.1})}{=}&\underline{D(P_{l})p_{i}}\otimes 1_{H'}\otimes \underline{U(p^{i})P^{l}}\otimes 1_{H'}\\
 &\stackrel{(\ref{eq:2.1})}{=}&1_H\otimes 1_{H'}\otimes 1_H\otimes 1_{H'}.
 \end{eqnarray*}
 While
 \begin{eqnarray*}
 {\tilde{\alpha}}_{i}\tilde{D}({\tilde{\alpha}^{-1}}_{j})\otimes \tilde{\alpha}^{-1j}\tilde{U}(\tilde{\alpha}^{i})
 &=&r_{k}\underline{p_{i}D(P_{l})}D(R_{n})\otimes p'_{j}Q^{t}D'(q^{s})D'(P'_{m})\\
 &&\otimes q_{s}\underline{P^{l}U(p^{i})}U(Q_{t})\otimes P'^{m}R^{n}U'(r^{k})U'(p'^{j})\\
 &\stackrel{(\ref{eq:2.1})}{=}&r_{k}D(R_{n})\otimes p'_{j}\underline{Q^{t}D'(q^{s})}D'(P'_{m})\\
 &&\otimes \underline{q_{s}U(Q_{t})}\otimes P'^{m}R^{n}U'(r^{k})U'(p'^{j})\\
 &\stackrel{(\ref{eq:3.1})}{=}&\underline{r_{k}D(R_{n})}\otimes p'_{j}D'(P'_{m})\otimes 1_H\otimes P'^{m}\underline{R^{n}U'(r^{k})}U'(p'^{j})\\
 &\stackrel{(\ref{eq:3.1})}{=}&1_H\otimes \underline{p'_{j}D'(P'_{m})}\otimes 1_H\otimes \underline{P'^{m}U'(p'^{j})}\\
 &\stackrel{(\ref{eq:2.1})}{=}&1_H\otimes 1_{H'}\otimes 1_H\otimes 1_{H'}.
 \end{eqnarray*}

 Secondly, since
 \begin{eqnarray*}
 (\tilde{D}\otimes \tilde{D})(\tilde{\alpha})
 &=&D(r_{k})D(p_{i})\otimes D'(p'_{j})D'(Q^{t})\\
 &&\otimes D(p^{i})D(Q_{t})\otimes D'(r^{k})D'(p'^{j})\stackrel{(\ref{eq:2.2})(\ref{eq:3.2})}{=}\tilde{\alpha},
 \end{eqnarray*}
 and $(\tilde{U}\otimes \tilde{U})(\tilde{\alpha})=\tilde{\alpha}$, Eq.(\ref{eq:2.2}) is satisfied.

 Thirdly, by Eqs.(\ref{eq:3.9}) and (\ref{eq:3.10}), we have
 \begin{eqnarray}
 &&p'_{j}\otimes r_{k}\otimes p'^{j}r^{k}=q^{s}p'_{j}Q^{t}\otimes q_{s}r_{k}Q_{t}\otimes r^{k}p'^{j};  \label{eq:3.11}\\
 &&p_{i}\otimes Q^{t}\otimes p^{i}Q_{t}=R_{n}p_{i}r_{k}\otimes R^{n}Q^{t}r^{k}\otimes Q_{t}p^{i}.  \label{eq:3.12}
 \end{eqnarray}
 Then
 \begin{eqnarray*}
 \tilde{\alpha}_{12}\tilde{\alpha}_{13}\tilde{\alpha}_{23}
 &\stackrel{}{=}&(r_{k}p_{i}\otimes p'_{j}Q^{t})(r_{\bar{k}}p_{\bar{i}}\otimes p'_{\bar{j}}Q^{\bar{t}})\otimes (p^{i}Q_{t}\otimes
 r^{k}p'^{j})(r_{\bar{\bar{k}}}p_{\bar{\bar{i}}}\otimes p'_{\bar{\bar{t}}}Q^{\bar{\bar{t}}})\\
 &&\otimes (p^{\bar{i}}Q_{\bar{t}}\otimes r^{\bar{k}}p'^{\bar{j}})(p^{\bar{\bar{i}}}Q_{\bar{\bar{t}}}\otimes r^{\bar{\bar{k}}}p'^{\bar{\bar{t}}})\\
 &\stackrel{}{=}&r_{k}p_{i}r_{\bar{k}}p_{\bar{i}}\otimes p'_{j}Q^{t}\underline{p'_{\bar{j}}}Q^{\bar{t}}\otimes p^{i}Q_{t}\underline{r_{\bar{\bar{k}}}}p_{\bar{\bar{i}}}\otimes
 r^{k}p'^{j}p'_{\bar{\bar{t}}}Q^{\bar{\bar{t}}}\\
 &&\otimes p^{\bar{i}}Q_{\bar{t}}p^{\bar{\bar{i}}}Q_{\bar{\bar{t}}}\otimes r^{\bar{k}}\underline{p'^{\bar{j}}r^{\bar{\bar{k}}}}p'^{\bar{\bar{t}}}\\
 &\stackrel{(\ref{eq:3.11})}{=}&r_{k}p_{i}r_{\bar{k}}p_{\bar{i}}\otimes p'_{j}\underline{Q^{t}q^{s}}p'_{\bar{j}}Q^{\bar{\bar{\bar{t}}}}Q^{\bar{t}}\otimes p^{i}\underline{Q_{t}q_{s}}r_{\bar{\bar{k}}}Q_{\bar{\bar{\bar{t}}}}p_{\bar{\bar{i}}}\otimes r^{k}p'^{j}p'_{\bar{\bar{t}}}Q^{\bar{\bar{t}}}\\
 &&\otimes p^{\bar{i}}Q_{\bar{t}}p^{\bar{\bar{i}}}Q_{\bar{\bar{t}}}\otimes r^{\bar{k}}r^{\bar{\bar{k}}}p'^{\bar{j}}p'^{\bar{\bar{t}}}\\
 &\stackrel{}{=}&r_{k}p_{i}r_{\bar{k}}p_{\bar{i}}\otimes p'_{j}\underline{p'_{\bar{j}}}Q^{t}Q^{\bar{t}}\otimes p^{i}r_{\bar{\bar{k}}}Q_{t}p_{\bar{\bar{i}}}\otimes r^{k}p'^{j}\underline{p'_{\bar{\bar{t}}}}Q^{\bar{\bar{t}}}\\
 &&\otimes p^{\bar{i}}Q_{\bar{t}}p^{\bar{\bar{i}}}Q_{\bar{\bar{t}}}\otimes r^{\bar{k}}r^{\bar{\bar{k}}}\underline{p'^{\bar{j}}p'^{\bar{\bar{t}}}}\\
 &\stackrel{(\ref{eq:2.7})}{=}&r_{k}p_{i}r_{\bar{k}}p_{\bar{i}}\otimes \underline{p'_{j}P'_{m}}p'_{\bar{j}}p'_{\bar{\bar{\bar{j}}}}Q^{t}Q^{\bar{t}}\otimes p^{i}r_{\bar{\bar{k}}}Q_{t}p_{\bar{\bar{i}}}\otimes
 r^{k}\underline{p'^{j}P'^{m}}p'_{\bar{\bar{t}}}p'^{\bar{\bar{\bar{j}}}}Q^{\bar{\bar{t}}}\\
 &&\otimes p^{\bar{i}}Q_{\bar{t}}p^{\bar{\bar{i}}}Q_{\bar{\bar{t}}}\otimes r^{\bar{k}}r^{\bar{\bar{k}}}p'^{\bar{\bar{t}}}p'^{\bar{j}}\\
 &\stackrel{}{=}& r_{k}p_{i}\underline{r_{\bar{k}}}p_{\bar{i}}\otimes p'_{\bar{j}}p'_{j}Q^{t}Q^{\bar{t}}\otimes p^{i}\underline{r_{\bar{\bar{k}}}}Q_{t}p_{\bar{\bar{i}}}\otimes
 r^{k}p'_{\bar{\bar{t}}}p'^{j}Q^{\bar{\bar{t}}}\\
 &&\otimes p^{\bar{i}}Q_{\bar{t}}p^{\bar{\bar{i}}}Q_{\bar{\bar{t}}}\otimes \underline{r^{\bar{k}}r^{\bar{\bar{k}}}}p'^{\bar{\bar{t}}}p'^{\bar{j}}\\
 &\stackrel{(\ref{eq:3.5})}{=}&r_{k}\underline{p_{i}P_{l}}r_{\bar{k}}p_{\bar{\bar{\bar{i}}}}p_{\bar{i}}\otimes p'_{\bar{j}}p'_{j}Q^{t}Q^{\bar{t}}\otimes \underline{p^{i}P^{l}}r_{\bar{\bar{k}}}p^{\bar{\bar{\bar{i}}}}Q_{t}p_{\bar{\bar{i}}}\otimes r^{k}p'_{\bar{\bar{t}}}p'^{j}Q^{\bar{\bar{t}}}\\
 &&\otimes p^{\bar{i}}Q_{\bar{t}}p^{\bar{\bar{i}}}Q_{\bar{\bar{t}}}\otimes r^{\bar{\bar{k}}}r^{\bar{k}}p'^{\bar{\bar{t}}}p'^{\bar{j}}\\
 &\stackrel{}{=}&r_{k}\underline{r_{\bar{k}}}p_{i}p_{\bar{i}}\otimes p'_{\bar{j}}p'_{j}Q^{t}Q^{\bar{t}}\otimes r_{\bar{\bar{k}}}p^{i}Q_{t}p_{\bar{\bar{i}}}\otimes
 r^{k}\underline{p'_{\bar{\bar{t}}}}p'^{j}Q^{\bar{\bar{t}}}\\
 &&\otimes p^{\bar{i}}Q_{\bar{t}}p^{\bar{\bar{i}}}Q_{\bar{\bar{t}}}\otimes r^{\bar{\bar{k}}}\underline{r^{\bar{k}}p'^{\bar{\bar{t}}}}p'^{\bar{j}}\\
 &\stackrel{(\ref{eq:3.6})}{=}&\underline{r_{k}R_{n}}r_{\bar{k}}r_{\bar{\bar{\bar{k}}}}p_{i}p_{\bar{i}}\otimes p'_{\bar{j}}p'_{j}Q^{t}Q^{\bar{t}}\otimes r_{\bar{\bar{k}}}p^{i}Q_{t}p_{\bar{\bar{i}}}\otimes
 \underline{r^{k}R^{n}}p'_{\bar{\bar{t}}}r^{\bar{\bar{\bar{k}}}}p'^{j}Q^{\bar{\bar{t}}}\\
 &&\otimes p^{\bar{i}}Q_{\bar{t}}p^{\bar{\bar{i}}}Q_{\bar{\bar{t}}}\otimes r^{\bar{\bar{k}}}p'^{\bar{\bar{t}}}r^{\bar{k}}p'^{\bar{j}}\\
 &\stackrel{}{=}&r_{\bar{k}}r_{k}p_{i}p_{\bar{i}}\otimes p'_{\bar{j}}p'_{j}Q^{t}\underline{Q^{\bar{t}}}\otimes r_{\bar{\bar{k}}}p^{i}Q_{t}\underline{p_{\bar{\bar{i}}}}\otimes
 p'_{\bar{\bar{t}}}r^{k}p'^{j}Q^{\bar{\bar{t}}}\\
 &&\otimes p^{\bar{i}}\underline{Q_{\bar{t}}p^{\bar{\bar{i}}}}Q_{\bar{\bar{t}}}\otimes r^{\bar{\bar{k}}}p'^{\bar{\bar{t}}}r^{\bar{k}}p'^{\bar{j}}\\
 &\stackrel{(\ref{eq:3.7})}{=}&r_{\bar{k}}r_{k}p_{i}p_{\bar{i}}\otimes p'_{\bar{j}}p'_{j}\underline{Q^{t}q^{s}}Q^{\bar{t}}Q^{\bar{\bar{\bar{t}}}}\otimes r_{\bar{\bar{k}}}p^{i}\underline{Q_{t}q_{s}}p_{\bar{\bar{i}}}Q_{\bar{\bar{\bar{t}}}}\otimes p'_{\bar{\bar{t}}}r^{k}p'^{j}Q^{\bar{\bar{t}}}\\
 &&\otimes p^{\bar{i}}p^{\bar{\bar{i}}}Q_{\bar{t}}Q_{\bar{\bar{t}}}\otimes r^{\bar{\bar{k}}}p'^{\bar{\bar{t}}}r^{\bar{k}}p'^{\bar{j}}\\
 &\stackrel{}{=}&r_{\bar{k}}r_{k}p_{i}p_{\bar{i}}\otimes p'_{\bar{j}}p'_{j}\underline{Q^{\bar{t}}}Q^{t}\otimes r_{\bar{\bar{k}}}p^{i}p_{\bar{\bar{i}}}Q_{t}\otimes
 p'_{\bar{\bar{t}}}r^{k}p'^{j}\underline{Q^{\bar{\bar{t}}}}\\
 &&\otimes p^{\bar{i}}p^{\bar{\bar{i}}}\underline{Q_{\bar{t}}Q_{\bar{\bar{t}}}}\otimes r^{\bar{\bar{k}}}p'^{\bar{\bar{t}}}r^{\bar{k}}p'^{\bar{j}}\\
 &\stackrel{(\ref{eq:3.8})}{=}&r_{\bar{k}}r_{k}p_{i}p_{\bar{i}}\otimes p'_{\bar{j}}\underline{p'_{j}P'_{m}}Q^{\bar{t}}p'_{\bar{\bar{\bar{j}}}}Q^{t}\otimes r_{\bar{\bar{k}}}p^{i}p_{\bar{\bar{i}}}Q_{t}\otimes
 p'_{\bar{\bar{t}}}r^{k}\underline{p'^{j}P'^{m}}Q^{\bar{\bar{t}}}p'^{\bar{\bar{\bar{j}}}}\\
 &&\otimes p^{\bar{i}}p^{\bar{\bar{i}}}Q_{\bar{\bar{t}}}Q_{\bar{t}}\otimes r^{\bar{\bar{k}}}p'^{\bar{\bar{t}}}r^{\bar{k}}p'^{\bar{j}}\\
 &\stackrel{}{=}&r_{\bar{k}}r_{k}p_{i}\underline{p_{\bar{i}}}\otimes p'_{\bar{j}}Q^{\bar{t}}p'_{j}Q^{t}\otimes r_{\bar{\bar{k}}}p^{i}\underline{p_{\bar{\bar{i}}}}Q_{t}\otimes
 p'_{\bar{\bar{t}}}r^{k}Q^{\bar{\bar{t}}}p'^{j}\\
 &&\otimes \underline{p^{\bar{i}}p^{\bar{\bar{i}}}}Q_{\bar{\bar{t}}}Q_{\bar{t}}\otimes r^{\bar{\bar{k}}}p'^{\bar{\bar{t}}}r^{\bar{k}}p'^{\bar{j}}\\
 &\stackrel{(\ref{eq:2.7})}{=}&r_{\bar{k}}r_{k}\underline{p_{i}P_{l}}p_{\bar{i}}p_{\bar{\bar{\bar{i}}}}\otimes p'_{\bar{j}}Q^{\bar{t}}p'_{j}Q^{t}\otimes  r_{\bar{\bar{k}}}\underline{p^{i}P^{l}}p_{\bar{\bar{i}}}p^{\bar{\bar{\bar{i}}}}Q_{t}\otimes p'_{\bar{\bar{t}}}r^{k}Q^{\bar{\bar{t}}}p'^{j}\\
 &&\otimes p^{\bar{\bar{i}}}p^{\bar{i}}Q_{\bar{\bar{t}}}Q_{\bar{t}}\otimes r^{\bar{\bar{k}}}p'^{\bar{\bar{t}}}r^{\bar{k}}p'^{\bar{j}}\\
 &\stackrel{}{=}&r_{\bar{k}}r_{k}\underline{p_{\bar{i}}}p_{i}\otimes p'_{\bar{j}}Q^{\bar{t}}p'_{j}Q^{t}\otimes r_{\bar{\bar{k}}}p_{\bar{\bar{i}}}p^{i}Q_{t}\otimes
 p'_{\bar{\bar{t}}}r^{k}\underline{Q^{\bar{\bar{t}}}}p'^{j}\\
 &&\otimes p^{\bar{\bar{i}}}\underline{p^{\bar{i}}Q_{\bar{\bar{t}}}}Q_{\bar{t}}\otimes r^{\bar{\bar{k}}}p'^{\bar{\bar{t}}}r^{\bar{k}}p'^{\bar{j}}\\
 &\stackrel{(\ref{eq:3.12})}{=}&r_{\bar{k}}\underline{r_{k}R_{n}}p_{\bar{i}}r_{\bar{\bar{\bar{k}}}}p_{i}\otimes p'_{\bar{j}}Q^{\bar{t}}p'_{j}Q^{t}\otimes r_{\bar{\bar{k}}}p_{\bar{\bar{i}}}p^{i}Q_{t}\otimes
 p'_{\bar{\bar{t}}}\underline{r^{k}R^{n}}Q^{\bar{\bar{t}}}r^{\bar{\bar{\bar{k}}}}p'^{j}\\
 &&\otimes p^{\bar{\bar{i}}}Q_{\bar{\bar{t}}}p^{\bar{i}}Q_{\bar{t}}\otimes r^{\bar{\bar{k}}}p'^{\bar{\bar{t}}}r^{\bar{k}}p'^{\bar{j}}\\
 &\stackrel{}{=}&r_{\bar{k}}p_{\bar{i}}r_{k}p_{i}\otimes p'_{\bar{j}}Q^{\bar{t}}p'_{j}Q^{t}\otimes r_{\bar{\bar{k}}}p_{\bar{\bar{i}}}p^{i}Q_{t}\otimes
 p'_{\bar{\bar{t}}}Q^{\bar{\bar{t}}}r^{k}p'^{j}\\
 &&\otimes p^{\bar{\bar{i}}}Q_{\bar{\bar{t}}}p^{\bar{i}}Q_{\bar{t}}\otimes r^{\bar{\bar{k}}}p'^{\bar{\bar{t}}}r^{\bar{k}}p'^{\bar{j}}\\
 &\stackrel{}{=}&(r_{\bar{k}}p_{\bar{i}}\otimes p'_{\bar{j}}Q^{\bar{t}})(r_{k}p_{i}\otimes p'_{j}Q^{t})\otimes (r_{\bar{\bar{k}}}p_{\bar{\bar{i}}}\otimes
 p'_{\bar{\bar{t}}}Q^{\bar{\bar{t}}})(p^{i}Q_{t}\otimes r^{k}p'^{j})\\
 &&\otimes (p^{\bar{\bar{i}}}Q_{\bar{\bar{t}}}\otimes r^{\bar{\bar{k}}}p'^{\bar{\bar{t}}})(p^{\bar{i}}Q_{\bar{t}}\otimes r^{\bar{k}}p'^{\bar{j}})\\
 &\stackrel{}{=}&\tilde{\alpha}_{23}\tilde{\alpha}_{13}\tilde{\alpha}_{12}.
 \end{eqnarray*}
 So, Eq.(\ref{eq:2.3}) is satisfied for $H\otimes H'$.  Thus ($H\otimes H'$, $\tilde{\alpha}$, $\tilde{D}$, $\tilde{U}$) is an OQA. \end{proof}

\subsection{OQA structures on the tensor product of two OQAs}

\begin{thm}\label{thm:3.6}
 Suppose that $(H, p, D, U)$, $(H', p', D', U')$ are OQAs and $(H, H', p, p', r,D, U, D',$ $ U')$ is an OQA nonuple, write $p=p_{i}\otimes p^{i}, p'=p'_{j}\otimes p'^{j}, P=p^{-1}=P_{l}\otimes P^{l}, P'=p'^{-1}=P'_{m}\otimes P'^{m}, r=r_{k}\otimes r^{k}, R=r^{-1}=R_{n}\otimes R^{n}$, then
 ($H\otimes H'$, $\tilde{\alpha}$, $\tilde{D}$, $\tilde{U}$) is an OQA, where
 $$
 ~~~~~\tilde{\alpha}=r_{k}p_{i}\otimes p'_{j}R^{n}\otimes p^{i}R_{n}\otimes r^{k}p'^{j},~\tilde{D}=D\otimes D',~\tilde{U}=U\otimes U';
 $$
 $$
 \hbox{and~~}\tilde{\alpha}^{-1}=P_{l}R_{n}\otimes r^{k}P'_{m}\otimes r_{k}P^{l}\otimes P'^{m}R^{n}.
 $$
\end{thm}

\begin{proof} Let $q=r$ in Theorem \ref{thm:3.5}. \end{proof}

\subsection{OQA structures on the tensor product of an OQA with itself}

\begin{thm}\label{thm:3.7}
 Suppose that $(H, p, D, U)$ is an OQA, and write $p=p_{i}\otimes p^{i}=p_{j}\otimes p^{j}=p_{k}\otimes p^{k}$, $p^{-1}=P=P_{l}\otimes P^{l}$, then $(H\otimes H, \tilde{\alpha}, \tilde{D}, \tilde{U})$ is an OQA, where
 $$
 ~~~~~\tilde{\alpha}=p_{i}p_{j}\otimes  p_{k}P^{l}\otimes  p^{j}P_{l}\otimes  p^{i}p^{k},~\tilde{D}=D\otimes D\hbox{~and~}\tilde{U}=U\otimes U.
 $$
\end{thm}

\begin{proof} Let $p'=r=p$ in Theorem \ref{thm:3.6}.      \end{proof}

\begin{rmk}\label{rmk:3.8}
 An OQA structure on the tensor product of an OQA with itself is derived in Theorem \ref{thm:3.7}, which is different from the one given in Theorem \ref{thm:2.3} (see \cite[Theorem 4.1]{Ra2}).
\end{rmk}

\subsection{Relation with Chen's result}

\begin{cor}\label{cor:3.9}
 Assume that $(H, p)$ and $(H', p')$ are quasitriangular Hopf algebras, where $p=p_{i}\otimes p^{i}\in H\otimes H,~p'=p'_{j}\otimes p'^{j}\in H'\otimes H'$.  Let $r=r_{k}\otimes r^{k}\in H\otimes H'$ be a weak $\mathfrak{R}$-matrix of $(H, H')$. Set $[p, p']=r_{k}p_{i}\otimes p'_{j}R^{n}\otimes p^{i}R_{n}\otimes r^{k}p'^{j}$, where $R=r^{-1}$, then $(H\otimes H', [p, p'], id_{H\otimes H'}, S_H^{-2}\otimes S_{H'}^{-2})$ is an OQA.
\end{cor}

\begin{proof}
 Since $(H, p)$ and $(H', p')$ are quasitriangular Hopf algebras, $(H, p, id_H, S_H^{-2})$ and $(H', p',$ $id_{H'}, S_{H'}^{-2})$ are OQAs.

 While
 \begin{eqnarray*}
 D(R_{n})r_{k}\otimes U'(r^{k})R^{n}
 &=&R_{n}\underline{r_{k}}\otimes S_{H'}^{-2}(\underline{r^{k}})R^{n}\\
 &\stackrel{(\ref{eq:2.16})}{=}&\underline{R_{n}}S_{H}^{2}(r_{k})\otimes r^{k}\underline{R^{n}}\\
 &\stackrel{(\ref{eq:2.15})}{=}&\underline{S_{H}(r_{\bar{k}})S_{H}^{2}(r_{k})}\otimes r^{k}r^{\bar{k}}\\
 &=&S_{H}(\underline{S_{H}(r_{k})}r_{\bar{k}})\otimes \underline{r^{k}}r^{\bar{k}}\\
 &\stackrel{(\ref{eq:2.15})}{=}&S_{H}(\underline{R_{n}r_{\bar{k}}})\otimes \underline{R^{n}r^{\bar{k}}}\\
 &=&S_{H}(1_H)\otimes 1_{H'}=1_H\otimes 1_{H'}.
 \end{eqnarray*}

 Similarly, we have $r_{k}R_{n}\otimes R^{n}S_{H'}^{-2}(r^{k})=1_H\otimes 1_{H'}$.  So,  Eq.(\ref{eq:3.1}) holds for $r$. And Eq.(\ref{eq:3.2}) is satisfied for $r$ by Eq.(\ref{eq:2.16}). Then, by Proposition \ref{pro:3.3}, we know $(H, H', p, p', r, id_H, S_{H}^{-2}$, $id_{H'}, S_{H'}^{-2})$ is an OQA nonuple.

 Thus $(H\otimes H', [p, p'], id_{H\otimes H'}, S_H^{-2}\otimes S_{H'}^{-2})$ is an OQA by Theorem \ref{thm:3.6}.     \end{proof}

\begin{thm}\label{thm:3.10}
 Under the assumption of Corollary \ref{cor:3.9}, $(H\bowtie^r H', [p, p'])$ is a quasitriangular Hopf algebra by Theorem \ref{thm:2.9}. Then $(H\otimes H', [p, p'], id_{H\otimes H'}, \bar{S}^{-2})$ is an OQA, and is equal to the OQA structure in Corollary \ref{cor:3.9}.
\end{thm}

\begin{proof}
 Since $\bar{S}(h\otimes h')=R(S_{H}(h)\otimes S_{H'}(h'))r=(S_H\otimes S_{H'})(r(h\otimes h')R)$,
 \begin{equation}\label{eq:3.13}
 \bar{S}^{-1}(h\otimes h')=R(S_{H}^{-1}(h)\otimes S_{H'}^{-1}(h'))r
 \end{equation}
 and
 \begin{equation}\label{eq:3.14}
 \bar{S}^{-1}(h\otimes h')=(S_{H}^{-1}\otimes S_{H'}^{-1})(r(h\otimes h')R).
 \end{equation}
 In fact,  we have
 \begin{eqnarray*}
 \bar{S}^{-1}(\bar{S}(h\otimes h'))
 &=&\bar{S}^{-1}(R(S_{H}(h)\otimes S_{H'}(h'))r)\\
 &=&(S_{H}^{-1}\otimes S_{H'}^{-1})(\bar{r}(R(S_{H}(h)\otimes S_{H'}(h'))r)\bar{R})\\
 &=&(S_{H}^{-1}\otimes S_{H'}^{-1})(S_{H}(h)\otimes S_{H'}(h'))\\
 &=&h\otimes h',
 \end{eqnarray*}
 \begin{eqnarray*}
 \bar{S}(\bar{S}^{-1}(h\otimes h'))
 &=&\bar{S}(R(S_{H}^{-1}(h)\otimes S_{H'}^{-1}(h'))r)\\
 &=&({S_{H}}\otimes {S_{H'}})(\bar{r}(R(S_{H}^{-1}(h)\otimes S_{H'}^{-1}(h'))r)\bar{R})\\
 &=&({S_{H}}\otimes {S_{H'}})(S_{H}^{-1}(h)\otimes S_{H'}^{-1}(h'))\\
 &=&h\otimes h'.
 \end{eqnarray*}
 Then
 \begin{eqnarray*}
 \bar{S}^{-2}(h\otimes h')
 &\stackrel{(\ref{eq:3.13})}{=}&\bar{S}^{-1}(R(S_{H}^{-1}(h)\otimes S_{H'}^{-1}(h'))r)\\
 &\stackrel{(\ref{eq:3.14})}{=}&(S_{H}^{-1}\otimes S_{H'}^{-1})(\bar{r}(R(S_{H}^{-1}(h)\otimes S_{H'}^{-1}(h'))r)\bar{R})\\
 &\stackrel{}{=}&(S_{H}^{-1}\otimes S_{H'}^{-1})(S_{H}^{-1}(h)\otimes S_{H'}^{-1}(h'))\\
 &\stackrel{}{=}&(S_{H}^{-2}\otimes S_{H'}^{-2})(h\otimes h')
 \end{eqnarray*}

 Therefore, $\bar{S}^{-2}=S_{H}^{-2}\otimes S_{H'}^{-2}$, and we finish the proof. \end{proof}

\begin{rmk}\label{rmk:3.11}
 By Theorem \ref{thm:3.10}, our results generalize the OQA construction implicit in Theorem \ref{thm:2.9} (see \cite[Theorem 2.2]{Ch}).
\end{rmk}

\section{Applciations}
 In this section, several examples are given. Especially, by Example \ref{ex:4.3}, we can get our results here is different from Radford's.

\begin{ex}\label{ex:4.1}
 With notation as Case I in Example \ref{ex:3.4}. By Theorem \ref{thm:3.6} and a tedious computation, we can get the OQA structure $\tilde{\alpha}$ on the tensor product $H\otimes H'$ of  $(H=M_2(K), p, t, t)$ and $(H'=M_3(K), p', t', t')$. Here we use $\textbf{0}_{9\times 9}$ to denote the $9\times 9$ zero matrix and $x=a-a^{-1}$. Then
 $$
 \tilde{\alpha}=\left(%
 \begin{array}{cccc}
  A_{11} & \textbf{0}_{9\times 9} & \textbf{0}_{9\times 9} & A_{14} \\
  \textbf{0}_{9\times 9} & \textbf{0}_{9\times 9} & A_{23} & \textbf{0}_{9\times 9} \\
  \textbf{0}_{9\times 9} & \textbf{0}_{9\times 9} & \textbf{0}_{9\times 9} & \textbf{0}_{9\times 9} \\
  A_{41} & \textbf{0}_{9\times 9} & \textbf{0}_{9\times 9} & A_{44} \\
 \end{array}%
 \right)_{36\times 36},
 $$
 where
 \begin{itemize}
  \item  in $A_{11}=(b_{ij})_{9\times 9}$, $b_{11}=b_{55}=b_{99}=a^2, b_{15}=1-a, b_{19}=b_{51}=b_{91}=-a, b_{24}=b_{37}=b_{68}=xa, b_{59}=a, b_{95}=a+1$ and other entries are 0;
  \item  in $A_{14}=(b_{ij})_{9\times 9}$, $b_{11}=b_{55}=-a, b_{19}=b_{51}=1, b_{24}=-x, b_{37}=b_{68}=x, b_{59}=b_{91}=-1, b_{99}=a$ and other entries are 0;
  \item  in $A_{23}=(b_{ij})_{9\times 9}$, $b_{11}=b_{55}=b_{99}=xa, b_{15}=b_{19}=b_{51}=b_{91}=-x, b_{24}=b_{37}=b_{68}=x^2, b_{59}=b_{95}=x$ and other entries are 0;
  \item  in $A_{41}=(b_{ij})_{9\times 9}$, $b_{11}=b_{45}=-a, b_{15}=b_{21}=b_{49}=b_{91}=1, b_{19}=b_{95}=-1, b_{24}=-x, b_{37}=b_{58}=x, b_{99}=a$ and other entries are 0;
  \item  in $A_{44}=(b_{ij})_{9\times 9}$, $b_{11}=b_{55}=b_{99}=a^2, b_{15}=b_{51}=b_{59}=b_{95}=-a, b_{19}=b_{91}=a, b_{24}=b_{37}=b_{68}=xa$ and other entries are 0.
 \end{itemize}

\end{ex}

\begin{ex}\label{ex:4.3}
 Let $(H=M_2(K), p, t, t)$ given in Example \ref{ex:3.4}. By Theorem \ref{thm:3.7} and a tedious computation, we can get the OQA structure $\tilde{\alpha}$ on the tensor product $H\otimes H$ of  $(H=M_2(K), p, t, t)$ with itself, where $\tilde{\alpha}$ is given as follows:

  $$
  \left(
    \begin{array}{cccccccccccccccc}
      a^2 & 0 & 0 & 1 & 0 & 0 & -x^2a & 0 & 0 & 0 & 0 & 0 & a^2 & 0 & 0 & 1 \\
       0 & 0 & xa &  0 & 0  &  0 & 0  &  0 & 0  &  0 & 0  &  0 & 0  &  0 & xa^{-1}  & 0 \\
      0 & 0  &  0 & 0  & -xa^2 &  0 & 0 & 0 & 0  &  0 & 0  &  0 & 0  &  0 & 0  &  0 \\
      a^2 & 0  &  0 & a^2  &  0  &  0 & 0  &  0 & 0  &  0 & 0  &  1  & 0 & 0 & 1 \\
      0 & 0 & x & 0 & 0 & 0 & 0  &  0 & xa & 0 & 0 & xa^{-1} & 0 & 0 & x(a^2-x^2) & 0 \\
      0 & 0 & 0 & 0  &  0 & 0 & 0 & 0 & 0  &  0 & x^2 & 0 & 0 & 0 & 0  &  0 \\
      0 & 0 & 0 & 0  &  0 & 0 & -x^2 & 0 & 0 & 0 & 0  &  0 & -x^2 & 0  &  0 & -x^2a \\
      0 & 0 & x & 0 & 0 & 0 & 0  &  0 & x & 0  &  0 & xa & 0  &  0 & x & 0 \\
      0 & 0 & 0 & 0 & 0 & 0 & 0 & 0 & 0 & 0 & 0 & 0 & 0 & 0 & 0 & 0 \\
      0 & 0 & 0 & 0 & 0 & 0 & 0 & 0 & 0 & 0 & 0 & 0 & 0 & 0 & 0 & 0 \\
      0 & 0 & 0 & 0 & 0 & 0 & 0 & 0 & 0 & 0 & 0 & 0 & 0 & 0 & 0 & 0 \\
      0 & 0 & 0 & 0 & 0 & 0 & 0 & 0 & 0 & 0 & 0 & 0 & 0 & 0 & 0 & 0 \\
      1 & 0 & 0 & 1 & 0 & 0 & -x^2a & 0 & 0 & 0 & 0 & 0 & a^2 & 0 & 0 & a^2 \\
      0 & 0 & xa & 0 & 0 & 0 & 0 & 0 & 0 & 0 & 0 & 0 & 0 & 0 & xa & 0 \\
      0 & 0 & 0 & 0 & -x & 0 & 0 & -xa^2 & 0 & 0 & 0 & 0 & 0 & 0 & 0 & 0 \\
      1 & 0 & 0 & a^2 & 0 & 0 & 0 & 0 & 0 & 0 & 0 & 0 & 1 & 0 & 0 & a^2
    \end{array}
  \right).
  $$
\end{ex}

\begin{rmk}\label{rmk:4.4}
 Example \ref{ex:4.3} implies that the OQA structure on the tensor product of an OQA with itself in Theorem \ref{thm:3.7} is different from the results in \cite{MW,Ra2}, since the $16\times 16$ matrix here is different from those.
\end{rmk}

\begin{ex}\label{ex:4.5}
 Let $H=K\langle g, x \rangle$ be an algebra with multiplication $g^2=1, x^2=0, gx=-xg$. Then $D=id_H: H\rightarrow H$ and $U:H'\rightarrow H'$ such that $U(1)=1, U(g)=g, U(x)=-x, U(gx)=xg$ are commuting algebra automorphisms of $H$ and $(H, p,
 D, U)$ is an OQA with
 $$
 p=p^1\otimes p^2=\frac{1}{2}(1\otimes 1+1\otimes g+g\otimes 1-g\otimes g)+\frac{\nu}{2}(x\otimes x+x\otimes gx+gx\otimes gx-gx\otimes x)
 $$
 for $\forall~\nu\in K$.

 Let $H'=K\langle t \rangle$ such that $t^2=1$, then $D'=U'=id_{H'}$ are commuting algebra automorphisms of $H'$ and $(H', p', D', U')$ is
 an OQA with
 $$
 p'=p'^1\otimes p'^2=\frac{1}{2}(1\otimes 1+1\otimes t+t\otimes 1-t\otimes t).
 $$

 Assume that $r=r^1\otimes r^2=\frac{1}{2}(1\otimes 1+1\otimes t+g\otimes 1-g\otimes t) \in H\otimes H'$, and a direct verification shows that $(H,H',p,p',r,D,U,D',U')$ is an
 OQA nonuple.

 By a straightforward computation and Theorem \ref{thm:3.7}, we get
 \begin{eqnarray*}
 &\tilde{\alpha}=\tilde{\alpha}^1\otimes \tilde{\alpha}^2=&\frac{1}{2}(1\otimes 1\otimes 1\otimes 1+1\otimes 1\otimes g\otimes t+g\otimes t\otimes 1\otimes 1-g\otimes t\otimes g\otimes t)\\
 &&+\frac{\nu}{2}(x\otimes t\otimes gx\otimes 1+x\otimes t\otimes x\otimes t+gx\otimes 1\otimes gx\otimes 1-gx\otimes 1\otimes x\otimes t ),
 \end{eqnarray*}
 for $\forall~\nu \in K$.
\end{ex}

 We end this paper with the following question.
\begin{ques}
 In \cite{Ra2}, D. E. Radford obtained the OQA structures on the tensor product of an OQA with itself through an algebra isomorphism of the quantum double $D(A)$ with $A\otimes A$ under certain condition. And there they also gave the application to a non-trivial ambient isotopy invariant of oriented links, which depend to a large extent on special quantum double $D(A)$.

 How can we apply our results in Theorem \ref{thm:3.7} (or more general Theorem \ref{thm:3.6}) to construct isotopy invariants without the help of the quantum double $D(A)$ ?
\end{ques}

 {\bf Acknowledgments:}
 This work was partially supported by China Postdoctoral Science Foundation (No. 2017M611291) and National Natural Science Foundation of China (Nos.11801150, 11601231).

 \end{document}